\documentclass[a4paper,12pt]{amsart}

\usepackage{url}
\usepackage{amssymb}
\usepackage{latexsym}
\usepackage{amsfonts}
\usepackage{amsmath}
\usepackage{eucal}
\usepackage{bm}
\usepackage{bbm}
\usepackage{graphicx}
\usepackage[english]{varioref}
\usepackage[nice]{nicefrac}
\usepackage[all]{xy}
\usepackage{amsthm}

\newcommand{\supp}{\text {\rm supp}}

\def\Om{\Omega}

\def\ge{\geqslant}
\def\le{\leqslant}
\def\a{\alpha}

\def\d{\delta}

\def\o{\omega}

\def\s{\sigma}
\def\t{\tau}
\def\th{\theta}
\def\k{\kappa}
\def\l{\lambda}

\def\i{^{-1}}
\def\ZZ{\mathbb Z}
\def\NN{\mathbb N}

\def\co{\mathcal O}

\def\tW{\tilde W}
\def\tw{\tilde w}
\def\tS{\tilde S}

\def\<{\langle}
\def\>{\rangle}

\theoremstyle{plain}
\newtheorem{thm}{Theorem}[section]
\newtheorem*{thm*}{Theorem}
 \newtheorem{prop}[thm]{Proposition}
 \newtheorem{lem}[thm]{Lemma}

\theoremstyle{definition}

\newtheorem*{rmk*}{Remark}
\newtheorem*{claim*}{Claim}

\begin{document}

\author{Xuhua He}
\address{Department of Mathematics, The Hong Kong University of Science and Technology, Clear Water Bay, Kowloon, Hong Kong}
\email{maxhhe@ust.hk}
\author{Zhongwei Yang}
\address{Department of Mathematics, The Hong Kong University of Science and Technology, Clear Water Bay, Kowloon, Hong Kong}
\email{yzw@ust.hk}
\title[Elements with finite Coxeter part in an affine Weyl group]{Elements with finite Coxeter part in an affine Weyl group}

\begin{abstract}Let $W_a$ be an affine Weyl group and $\eta:W_a\longrightarrow W_0$ be the natural projection to the corresponding finite Weyl group. We say that $w\in W_a$ has finite Coxeter part if $\eta(w)$ is conjugate to a Coxeter element of $W_0$. The elements with finite Coxeter part is a union of conjugacy classes of $W_a$. We show that for each conjugacy class $\co$ of $W_a$ with finite Coxeter part there exits a unique maximal proper parabolic subgroup $W_J$ of $W_a$, such that the set of minimal length elements in $\co$ is exactly the set of Coxeter elements in $W_J$. Similar results hold for twisted conjugacy classes.
\end{abstract}

\maketitle
%\tableofcontents

\section*{Introduction}
In \cite{GP}, Geck and Pfeiffer showed that elements of minimal length in the conjugacy classes of finite Weyl groups play a quite special role. The results on minimal length elements have a lot of applications in representation theory of finite Hecke algebra and algebraic groups, as well as the geometry of unipotent classes.

Recently, the first author, joint with Nie \cite{He2}, \cite{HN} studied minimal length elements in the conjugacy classes of affine Weyl groups and showed that these elements also play a special role. It is expected that the minimal length elements will have applications in representation theory of affine Hecke algebra and p-adic groups, as well as reduction of Shimura varieties.

Although the proof of \cite{HN} is case-free, it is still useful to have concrete data available for each conjugacy class. In this paper, we study some special conjugacy classes of affine Weyl groups and give an explicit description of the minimal length elements in these conjugacy classes. We show that the minimal length elements in a conjugacy class with finite Coxeter part (see $\S$\ref{finCox} for the precise definition) are exactly the Coxeter elements for a unique maximal proper parabolic subgroup of the affine Weyl group. The precise statement for this ``Coxeter=Coxeter'' theorem is Theorem \ref{main}.

This result is also a necessary ingredient of in the study of dimension formula of affine Deligne-Lusztig varieties. See \cite{GH} and \cite{H99}.

\section{The main Theorem}

\subsection{}\label{setup1} Let $S$ be a finite set and $(m_{i j})_{i, j \in S}$ be a matrix with entries in $\mathbb N \cup \{\infty\}$ such that $m_{i i}=1$ and $m_{i j}=m_{j i} \ge 2$ for all $i \neq j$. Let $W$ be a group defined by generators $s_i$ for $i \in S$ and relations $(s_i s_j)^{m_{i j}}=1$ for $i, j \in S$ with $m_{i j}< \infty$. We say that $(W, S)$ is a {\it Coxeter group}. Sometimes we just call $W$ itself a Coxeter group.

Let $H$ be a group of automorphisms of the group $W$ that preserves $S$. Set $W'=W\rtimes H$. Then an element in $W'$ is of the form $w \d$ for some $w \in W$ and $\d \in H$. We have that $(w \d) (w' \d')=w \d(w') \d \d' \in W'$ with $\d, \d' \in H$.

For $w \in W$ and $\d \in H$, we set $\ell(w \d)=\ell(w)$, where $\ell(w)$ is the length of $w$ in the Coxeter group $(W,S)$. Thus $H$ consists of length $0$ elements in $W'$.

For $J \subset S$, we denote by $W_J$ the standard parabolic subgroup of $W$ generated by $s_j$ for $j\in J$ and by $W^J$ (resp. ${}^J W$) the set of minimal coset representatives in $W/W_J$ (resp. $W_J \backslash W$).
%For $J, K \subset S$, we simply write $W^J \cap {}^K W$ as ${}^K W^J$.

For $J \subset S$ with $W_J$ finite, we denote by $w^J_0$ the maximal element in $W_J$.

\subsection{} For $w \in W$, we denote by $\supp(w)$ the set of $i\in S$ such that $s_i$ appears in some (or equivalently, any) reduced expression of $w$. For $w \in W$ and $\d \in H$, we set $\supp(w \d)=\cup_{n \in \ZZ} \d^n(\supp(w))$. Then $\supp(w \d)$ is the minimal $\d$-stable subset $J$ of $S$ such that $w \d \in W_J \rtimes \<\d\> \subset W'$.

We follow \cite[7.3]{Sp}. Let $\d \in H$. For each $\delta$-orbit in $S$, we pick a simple reflection. Let $g$ be the product of these simple reflections (in any order) and put $c=g \d \in W'$. We call $c$ a {\it Coxeter element} of $W'$. Then $\supp(c)=S$ for any Coxeter element $c$ of $W'$.

\subsection{} Let $\Phi$ be an irreducible reduced root system and $W_0$ be the corresponding finite Weyl group. Then $(W_0, S_0)$ is a Coxeter group, where $S_0=$\{$i$ : $s_i$ is a simple reflection in $W_0$\}.

Let $P^\vee$ be the coweight lattice and $Q^\vee$ be the coroot lattice. Let $$W_a=Q^\vee \rtimes W_0=\{t^\chi w; \chi \in Q^\vee, w \in W_0\}$$ be the associated affine Weyl group and $$\tW=P^\vee \rtimes W_0=\{t^\chi w; \chi \in P^\vee, w \in W_0\}$$ be the associated extended affine Weyl group. The multiplication is given by the formula $(t^\chi w) (t^{\chi'} w')=t^{\chi+w \chi'} w w'$.

Set $\tS=S_0 \cup \{0\}$ and $s_0=t^{\th^\vee} s_\th$, where $\th$ is the corresponding largest positive root. Then $W_a$ is a normal subgroup of $\tW$ and is a Coxeter group with generators $s_i$ (for $i \in \tS$).

Following \cite{IM65}, we define the length function on $\tW$ by $$\ell(t^\chi w)=\sum_{\a \in \Phi^+, w \i(\a) \in \Phi^+} |\<\chi, \a\>|+\sum_{\a \in \Phi^+, w \i(\a) \in \Phi^-} |\<\chi, \a\>-1|.$$ For any coset of $W_a$ in $\tW$, there is a unique element of length $0$. Moreover, there is a natural group isomorphism between $\Om=\{\t \in \tW; \ell(\t)=0\}$ and $\tW/W_a \cong P^\vee/Q^\vee$.

\subsection{} Let $\d$ be a diagram automorphism of $(W_0, S_0)$ and $\<\d\>$ be the group of automorphisms on $W_0$ generated by $\d$. Set $$W_0'=W_0 \rtimes \<\d\>.$$ Notice that $\d$ induces natural actions on $Q^\vee$, $P^\vee$, $W_a$ and $\tW$, which we still denote by $\d$. It also gives a bijection on $\tS$ which sends $S_0$ to $S_0$ and sends $0 \in \tS$ to $0$. Set $$\tW'=P^\vee \rtimes W_0'=\tW \rtimes \<\d\>.$$ Then $\Om'=\Om \rtimes \<\d\>$ is the set of length $0$ elements in $\tW'$ and $\tW'=W_a \rtimes \Om'$.

\subsection{} Define the action of $W_0$ on $W_0'$ by $w \cdot w'=w w' w \i$. Each orbit of $W_0$ is called a {\it $W_0$-conjugacy class} of $W_0'$. We define $W_a$-conjugacy classes and $\tW$-conjugacy classes of $\tW'$ in the same way. Notice that $W_a$ is a normal subgroup of $\tW'$. Thus each $W_a$-conjugacy class of $\tW'$ is contained in $W_a \t$ for some $\t \in \Om'$.

Let $\eta: \tW' \to W_0'$ be the projection map, i.e., $\eta(t^\chi w)=w$ for any $\chi \in P^\vee$ and $w \in W'_0$. For any $\tw \in \tW'$, we call $\eta(\tw)$ the {\it finite part} of $\tw$.

It is easy to see that $\eta$ sends a $\tW$-conjugacy class of $\tW'$ to a $W_0$-conjugacy class of $W'_0$.

\subsection{}\label{finCox} It is known that any two Coxeter elements of $W'_0$ in the same coset $W'_0/W_0$ are conjugated by an element of $W_0$.

Let $\co$ be a $W_a$-conjugacy class of $\tW'$ and $\co'$ be a $\tW$-conjugacy class of $\tW'$. We say that $\co$ (resp. $\co'$) {\it has finite Coxeter part} if $\eta(\co)$ (resp. $\eta(\co')$) contains a Coxeter element of $W'_0$. The purpose of this paper is to give an explicit description of the minimal length element in $\co$. We prove the following ``Coxeter=Coxeter'' theorem.

\begin{thm}\label{main}
Let $\co$ be a $W_a$-conjugacy class of $\tW'$ with finite Coxeter part and $\co_{\min}$ be the set of minimal length elements in $\co$. Let $\t \in \Om'$ with $\co \subset W_a \t$. Then there exists a unique maximal proper $\t$-stable subset $J$ of $\tS$ such that $\co_{\min}$ is the set of Coxeter elements of $W_J \rtimes \<\t\>$ that are contained in $W_J \t \subset W_J \rtimes \<\t\>$. Here we embed $W_J \rtimes \<\t\>$ into $\tW'$ in a natural way.
\end{thm}

\begin{rmk*}
For type $A$, it is first proved by the first author in \cite{He2}.
\end{rmk*}

\subsection{} Before proving the theorem, we first explain why a $W_a$-conjugacy class of $\tW'$ with finite Coxeter part does not contain a Coxeter element of $\tW'$ and hence why we need proper subset of $\tS$ in the theorem. Although it is not needed in the proof, it serves as a motivation for the theorem.

Let $t^\chi w \in \co$ with $\chi \in P^\vee$ and $w$ a finite Coxeter element of $W'_0$. Let $n$ be the order of $w$ in $W'_0$. It is known that the action of $1-w$ on $P^\vee \otimes_{\mathbb Q} {\mathbb C}$ is invertible. Hence $$(t^\chi w)^n=t^{\chi+w \chi+\cdots+w^{n-1} \chi} w^n=t^{\frac{1-w^n}{1-w} \chi}=1.$$ Therefore $t^\chi w$ is of finite order and hence any element in $\co$ is of finite order.

On the other hand, it is proved in \cite[Theorem 1]{Spe} (for untwisted case) and \cite[Proposition 3.1]{HN} that any Coxeter element of $\tW'$ is of infinite order. Hence $\co$ doesn't contain a Coxeter element of $\tW'$.

\section{existence of $J$}

\subsection{} Let $\co'$ be a $\tW$-conjugacy class of $\tW'$. Then $\displaystyle{\co'=\bigsqcup^r_{i=1}\co_i}$ is a disjoint union of $W_a$-conjugacy classes of $\tW'$. Since $\tW=W_a\rtimes\Om$, $\Om$ acts transitively on $\{\co_1, \co_2,\cdots,\co_r\}$. Moreover, if $\co_i=\t\co_j\t^{-1}$ for some $\t\in\Om$, then $(\co_i)_{min}=\t(\co_j)_{min}\t^{-1}$.
\subsection{}Let $\co'$ be a $\tW$-conjugacy class of $\tW'$ with finite Coxeter part, and let $\co$ be a $W_a$-conjugacy class of $\tW'$ with $\co\subset\co'$. The main purpose of this section is to show the ``existence'' part of the Theorem 1.1 for $\co'$ instead of $\co$. More precisely, there exists $\t \in \Om'$ and a maximal proper $\t$-stable subset $J$ of $\tS$ and a Coxeter element $c_J$ of $W_J\rtimes\langle\t\rangle $ such that $c_J\in\co'$.

By $\S2.1$, there exists $\s\in\Om$ such that $\s c_J\s^{-1}\in\co$. It is easy to see that $\s c_J\s^{-1}$ is a Coxeter element of $W_{\s(J)}\rtimes\langle \s\t\s^{-1} \rangle $. Thus the ``existence'' part of the theorem for $\tW$-conjugacy class $\co'$ deduces the ``existence'' part of it for $W_a$-conjugacy class $\co$.

Compared with $W_a$-conjugacy classes, it is much easier to classify $\tW$-conjugacy classes with finite Coxeter part and to find representatives. This is the reason that we consider $\tW$-conjugacy classes instead of $W_a$-conjugacy classes in this section.

\subsection{} We identify $\tW/W_a$ with $P^\vee/Q^\vee$ in the natural way. Let $\d$ be a diagram automorphism of $(W_0, S_0)$. Then $\<\d\>$ acts on $\tW/W_a \cong P^\vee/Q^\vee$. Let $(P^\vee/Q^\vee)_\d$ be the $\d$-coinvariant of $P^\vee/Q^\vee$. Let $$\k_\d: \tW\d \to(P^\vee/Q^\vee)_\d, \qquad w \mapsto w \d \i W_a$$ be the natural projection. We call $\k_\d$ the {\it Kottwitz map}.

The following result classifies the $\tW$-conjugacy classes of $\tW'$ with finite Coxeter part.

%Let $P^\vee=\sum^{n}_{i=1}\mathbb{Z}\o^\vee_i$ and $Q^\vee=\sum^{n}_{i=1}\mathbb{Z}\alpha^\vee_i$, where $\o^\vee_i$ is the associated fundamental coweight and $\alpha^\vee_i$ is the associated simple coroot. Let $\d$ be an outer diagram automorphism defined as in section 1.We define $\k_\d: \tW\d\longrightarrow(\tW/W_a)\d\cong(P^\vee/Q^\vee)\d\longrightarrow(P^\vee/Q^\vee)_\d$ to be a natural projection. Then we have the following proposition.
\begin{prop}\label{class}
We keep the assumption as above. Let $\co_0 \subset W_0 \d$ be a $W_0$-conjugacy class containing a Coxeter element of $W'_0$. Then for any $v \in(P^\vee/Q^\vee)_{\d}$, $\eta^{-1}(\co_0)\bigcap\k^{-1}_\d(v)$ is a single $\tW$-congugacy class of $\tW'$.
\end{prop}

\begin{proof} Let $\mu \in P^\vee$ such that the image of $\mu$ under the map $P^\vee \to (P^\vee/Q^\vee)_{\d}$ is $v$. Let $c \d \in \co_0$. Then $t^\mu c \d \in \eta^{-1}(\co_0)\cap\k^{-1}_\d(v)$. It is easy to see that $\eta^{-1}(\co_0)\bigcap\k^{-1}_\d(v)$ is a union of $\tW$-conjugacy classes. Now we prove that $\tW$ acts transitively on $\eta^{-1}(\co_0)\bigcap\k^{-1}_\d(v)$.

Let $\mu' \in P^\vee$ and $c' \d \in \co_0$ such that $t^{\mu'} c' \d\in \eta^{-1}(\co_0)\bigcap\k^{-1}_\d(v)$. Then after conjugating by a suitable element of $W_0$, we may assume that $c'=c$. By definition, $\mu' \in \mu+(1-\d) P^\vee+Q^\vee$. Thus it suffices to show that

(a) $(1-\d) P^\vee+Q^\vee=(1-c \d) P^\vee$.

For any $\l \in P^\vee$, $(1-c \d) \l=(1-\d) \l+(1-c) \d(\l) \in (1-\d) P^\vee+Q^\vee$. Hence $(1-\d) P^\vee+Q^\vee \supset (1-c \d) P^\vee$.

We first prove that

(b) $Q^\vee \subset (1-c \d) P^\vee$.

We may assume that $c=s_{i_1} s_{i_2} \cdots s_{i_k}$. Since $c \d$ is a Coxeter element of $W'_0$, $\d$-orbits on $S_0$ are \[\{i_1, \d(i_1), \cdots, \d^{r_1}(i_1)\}, \{i_2, \d(i_2), \cdots, \d^{r_2}(i_2)\}, \cdots, \{i_k, \d(i_k), ... , \d^{r_k}(i_k)\}.\]

For $1 \le j \le k$, \begin{align*} (1-c\d)(\o^\vee_{i_j}+\o^\vee_{\d(i_j)}+\cdots+\o^\vee_{\d^{r_j}(i_j)}) &=(1-c)(\o^\vee_{i_j}+\o^\vee_{\d(i_j)}+\cdots+\o^\vee_{\d^{r_j}(i_j)}) \\ &=(1-c) \o^\vee_{i_j}=s_{i_1}s_{i_2}...s_{i_{j-1}}\a^\vee_{i_j}.\end{align*} Therefore $\{\a^\vee_{i_1}, \a^\vee_{i_2}, ... , \a^\vee_{i_k}\}\subset(1-c\d)P^\vee$.

For any $m \in S_0$, \[(1-c \d) \a^\vee_m=\a^\vee_m-\a^\vee_{\d(m)}+(1-c) \a^\vee_{\d(m)} \in \a^\vee_m-\a^\vee_{\d(m)}+\sum_{1 \le j \le k} \ZZ \a^\vee_{i_j}.\] Thus $\a^\vee_m-\a^\vee_{\d(m)} \in (1-c \d) P^\vee$ for all $m \in S_0$. Hence for $1 \le j \le k$ and $n \in \NN$, one may show by induction that $\a^\vee_{\d^n(i_j)} \in (1-c \d) P^\vee$.

(b) is proved.

Now $(1-c \d) P^\vee/Q^\vee=(1-\d) P^\vee/Q^\vee$. Thus (a) is proved.
\end{proof}

\subsection{} %Notice that each coset $P^\vee/Q^\vee$ is represented by the zero coweight or a minuscule coweight. Let $\co'$ be a $\tW$-conjugacy class of $\tW'$ with finite Coxeter part. By Proposition \ref{class}, $\co$ contains an element of the form $t^\o c$, where $\o$ is either $0$ or a minuscule coweight and $c$ is a Coxeter element in $W'_0$. We call such an element a {\it standard representative element} of $\co'$.

In order to prove the ``existence'' part of Theorem 1.1 for $\co'$, we need the following key lemma which will be proved in section 3 via a case-by-case analysis.

\begin{lem}\label{key}
Let $\d'$ be a diagram automorphism of $(W_0, S_0)$ and $\t=t^{\o^\vee_i}w^{S_0-\{i\}}_0w^{S_0}_0$, where $\o^\vee_i$ is a minuscule coweight. Then there exists a maximal proper $\t \d'$-stable subset $J$ of $\tS$ and $c \in W_0$ such that $\supp(\t \d' c)=J$ and $w^{S_0-\{i\}}_0w^{S_0}_0 \d'(c) \d'$ is conjugate to a Coxeter element of $W'_0$.
\end{lem}

\subsection{} Now we prove the ``existence'' part of Theorem 1.1 for $\co'$.

Let $\t \in \Om$ and $\d' \in \<\d\>$ such that $\co' \cap W_a \t \d' \neq \emptyset$. If $\t=1$, then we may take $J=S_0$ and $c_J$ be any Coxeter element of $W_0 \d' \subset W'_0$.

If $\t \neq 1$, then $\t=t^{\o^\vee_i}w^{S_0-\{i\}}_0w^{S_0}_0$ for some minuscule coweight $\o^\vee_i$.  We take $J$ and $c$ from Lemma \ref{key}. Then $\t \d' c$ is a Coxeter element of $W_J \rtimes \<\t \d'\>$ and $\eta(\t \d' c)=w^{S_0-\{i\}}_0w^{S_0}_0 \d'(c) \d'$ is conjugate to a Coxeter element of $W'_0$. By Proposition \ref{class}, $\t \d' c \in \co'$.

%In section 4, we explain the uniqueness of $J$ and $\co_{min}=$\{Coxeter element of $W_J\rtimes\langle\t \rangle $\}.
\section{the key lemma}
%\subsection{}Let $\t\in\Om'$ and $\s\in\Om$, then $\t=\s$ or $\t=\s\d$. Follows \cite{IM65}, if $\t=\s$, then $\t=t^{\o^\vee_i}w^{I-\{i\}}_0w^I_0$; And if $\t=\s\d$, then $\t=t^{\o^\vee_i}w^{I-\{i\}}_0w^I_0\d$ where $\o^\vee_i$ is a minuscule coweight. To prove the lemma, it is sufficient to find $c\in\ W_a$ such that $c'=w^{I-\{i\}}_0w^I_0c$ is conjugate to a Coxeter element of $W_0$(without twist), or to find $c$ such that $c'=w^{I-\{i\}}_0w^I_0\d(c)\d$ is conjugate to a Coxeter element of $W_0\d$(with twist $\d$). In this section, we prove Lemma 2.2 by giving a case by case analysis.
%\subsection{}
%\
In this section, we verify Lemma \ref{key}. We use the same labeling of Dynkin diagram as in \cite{Bo}.

\

{\bf Type $A_{n-1}$}

This case was proved by the first author in \cite[Lemma 5.1]{He2}.

\

{\bf Type ${}^2 A_{n-1}$}

We may regard $\d'$ as the permutation $w_0^{S_0}=(1\ n) (2\ \ n-1) \cdots$ in $S_n$ and regard $W_0 \d' \subset W_0'$ as $S_n$. Under this identification, the $W_0$ conjugacy class that contains a Coxeter element in $W_0 \d'$ is the set of $n$-cycles when $n$ is odd and is the set of $n-1$ cycles when $n$ is even.

Let $\t=\t_i$. Then $\t \d'$-orbits on $\tS$ are $\{0, i\}$, $\{j,i-j\}$ for $0<j<i$, and $\{i+j,n-j\}$ for $0<j<n-i$.

We have the following four different cases:

Case 1: $n$ is odd and $i$ is odd.

In this case, we take $J=\tS-\{\frac{n+i}{2}\}$ and $c=s_{\frac{i+1}{2}}s_{\frac{i+3}{2}} \cdots s_{\frac{n+i}{2}-1}$. Then $w_0^{S_0-\{i\}}c$ is an $n$-cycle. In other words, $w^{S_0-\{i\}}_0w^{S_0}_0 \d'(c) \d'$ is conjugate to a Coxeter element in $W_0 \d'$.

Case 2: $n$ is odd and $i$ is even.

In this case, we take $J=\tS-\{\frac{i}{2}\}$ and $c=s_{\frac{i}{2}+1}s_{\frac{i}{2}+2} \cdots s_{\frac{n+i-1}{2}}$. Then $w_0^{S_0-\{i\}}c$ is an $n$-cycle. In other words, $w^{S_0-\{i\}}_0w^{S_0}_0 \d'(c) \d'$ is conjugate to a Coxeter element in $W_0 \d'$.

Case 3: $n$ is even and $i$ is odd.

In this case, we take $J=\tS-\{\frac{i-1}{2}, \frac{i+1}{2}\}$ and $c=s_{\frac{i+3}{2}}s_{\frac{i+5}{2}} \cdots s_{\frac{n+i-1}{2}}$. Then $w_0^{S_0-\{i\}}c$ is an $n-1$ cycle. In other words, $w^{S_0-\{i\}}_0w^{S_0}_0 \d'(c) \d'$ is conjugate to a Coxeter element in $W_0 \d'$.

Case 4: $n$ is even and $i$ is even.

In this case, we take $J=\tS-\{\frac{n+i}{2}\}$ and $c=s_{\frac{i}{2}}s_{\frac{i}{2}+1} \cdots s_{\frac{n+i}{2}-1}$. Then $w_0^{S_0-\{i\}}c$ is an $n-1$ cycle. In other words, $w^{S_0-\{i\}}_0w^{S_0}_0 \d'(c) \d'$ is conjugate to a Coxeter element in $W_0 \d'$.

\

{\bf Type $B_n$}

There is only one minuscule coweight: $\o^\vee_1$. So $\t=\t_1$. Now $\t$-orbits on $\tilde S$ are $\{0, 1\}$ and $\{i\}$ for $2\leqslant i\leqslant n$. We take $J=\tS-\{n\}$ and $c=s_1s_2 \cdots s_{n-1}$. Then $w^{S_0-\{1\}}_0w^{S_0}_0 c$ is conjugate to a Coxeter element of $W_0$.\\
\

{\bf Type $C_n$}

There is only one minuscule coweight: $\o^\vee_n$. So $\t=\t_n$. Now $\t$-orbits on $\tilde S$ are $\{i, n-i\}$ for $0\leqslant i\leqslant n$.

If $n$ is odd, we take $J=\tS-\{0, n\}$ and $c=s_{\frac{n+1}{2}}s_{\frac{n+3}{2}} \cdots s_{n-1}$. Then $w^{S_0-\{n\}}_0w^{S_0}_0 c$ is conjugate to a Coxeter element of $W_0$.

If $n$ is even, we take $J=\tS-\{\frac{n}{2}\}$ and $c=s_{\frac{n}{2}+1}s_{\frac{n}{2}+2} \cdots s_n$. Then $w^{S_0-\{n\}}_0w^{S_0}_0 c$ is conjugate to a Coxeter element of $W_0$.

\

{\bf Type $D_n$}

There are three minuscule coweights: $\o^\vee_1,\ \o^\vee_{n-1},\ \o^\vee_n$. There is an outer diagram automorphism of $D_n$ permuting the last two coweights. Thus it suffices to consider the case where $\t=\t_1$ or $\t_n$.

Case 1: $\t=\t_1$.

The $\t$-orbits on $\tilde S$ are $\{0, 1\},\ \{n-1, n\}$ and $\{i\}$ for $2 \le i \le n-2$. We take $J=\tS-\{n-1, n\}$ and $c=s_1s_2 \cdots s_{n-2}$. Then $w^{S_0-\{1\}}_0w^{S_0}_0 c=s_1s_2 \cdots s_n$ is a Coxeter element of $W_0$.

Case 2: $\t=\t_n$ and $n$ is odd.

The $\t$-orbits on $\tilde S$ are $\{0, n, 1, n-1\}$ and $\{i, n-i\}$ for $2 \le i \le \frac{n-1}{2}$. We take $J=\tS-\{\frac{n-1}{2}, \frac{n+1}{2}\}$ and $c=s_{\frac{n+3}{2}}s_{\frac{n+5}{2}} \cdots s_{n-2}s_n$. Then $w^{S_0-\{n\}}_0w^{S_0}_0 c$ is conjugate to a Coxeter element of $W_0$.

Case 3: $\t=\t_n$ and $n$ is even.

The $\t$-orbits on $\tilde S$ are $\{i, n-i\}$ for $0 \le i \le \frac{n}{2}$. We take $J=\tS-\{0, n\}$ and $c=s_{\frac{n}{2}}s_{\frac{n}{2}+1} \cdots s_{n-1}$. Then $w^{S_0-\{n\}}_0w^{S_0}_0 c$ is conjugate to a Coxeter element of $W_0$.

\

{\bf Type ${}^2 D_n$}

As explained above, it suffices to consider the following three cases.

Case 1: $\t=\t_1$.

The $\t \d'$-orbits on $\tilde S$ are $\{0, 1\}$ and $\{i\}$ for $2 \le i \le n$. We take $J=\tS-\{n\}$ and $c=s_1s_2 \cdots s_{n-2}s_{n-1}$. Then $w^{S_0-\{1\}}_0w^{S_0}_0 \d'(c) \d'=s_1s_2 \cdots s_{n-2} s_{n-1} \d'$ is a Coxeter element in $W_0 \d'$.

Case 2: $\t=\t_n$ and $n$ is odd.

The $\t \d'$-orbits on $\tilde S$ are $\{i, n-i\}$ for $0 \le i \le \frac{(n-1)}{2}$. We take $J=\tS-\{0, n\}$ and $c=s_{\frac{n+1}{2}}s_{\frac{n+3}{2}} \cdots s_{n-2} s_{n-1}$. Then $w^{S_0-\{n\}}_0w^{S_0}_0 \d'(c) \d'$ is conjugate to a Coxeter element in $W_0 \d'$.

Case 3: $\t=\t_n$ and $n$ is even.

The $\t \d'$-orbits on $\tilde S$ are $\{0, 1, n-1, n\}$ and $\{i, n-i\}$ for $2 \le i \le \frac{n}{2}$. We take $J=\tS-\{\frac{n}{2}\}$ and $c=s_{\frac{n}{2}+1}s_{\frac{n}{2}+2} \cdots s_{n-2} s_{n-1}$. Then $w^{S_0-\{n\}}_0w^{S_0}_0 \d'(c) \d'$ is conjugate to a Coxeter element in $W_0 \d'$.

\

{\bf Type ${}^3 D_4$}

Without loss of generality, we may assume that $\d'$ is the outer diagram automorphism on $D_4$ sending $s_1$ to $s_3$, $s_3$ to $s_4$ and $s_4$ to $s_1$. As $\<\d'\>$ acts transitively on $\{1, 3, 4\}$, it suffices to consider the case where $\t=\t_1$.

In this case, the $\t \d'$-orbits on $\tS$ are $\{0, 1, 4\}, \{2\}, \{3\}$. We take $J=\tS-\{3\}$ and $c=s_2 s_1$, then $w^{S_0-\{1\}}_0w^{S_0}_0 \d'(c) \d'$ is conjugate to a Coxeter element in $W_0 \d'$.

\

{\bf Type $E_6$}

There are two minuscule coweights: $\o^\vee_1$ and $\o^\vee_6$. The unique outer diagram automorphism of $E_6$ permutes these two coweights. Thus it suffices to consider the case where $\t=\t_1$. In this case, $\t$-orbits on $\tilde S$ are $\{0, 1, 6\}, \{2, 3, 5\}, \{4\}$. We take $J=\tS-\{0, 1, 6\}$ and $c=s_4 s_5$. Then $w^{S_0-\{1\}}_0w^{S_0}_0 c$ is conjugate to a Coxeter element of $W_0$.

\

{\bf Type ${}^2 E_6$}

As explained above, it suffices to consider the case where $\t=\t_1$. In this case, $\t \d'$-orbits on $\tilde S$ are $\{0, 1\}, \{2, 3\}, \{4\}, \{5\},  \{6\}$. We take $J=\tS-\{6\}$ and $c=s_5 s_4 s_3 s_1$. Then $w^{S_0-\{1\}}_0w^{S_0}_0 \d'(c) \d'$ is conjugate to a Coxeter element of $W'_0$.

\

{\bf Type $E_7$}

There is a unique minuscule coweight: $\o^\vee_7$. So $\t=\t_7$. In this case, $\t$-orbits on $\tilde S$ are $\{0, 7\}, \{1, 6\}, \{3, 5\}, \{2\}, \{4\}$. We take $J=\tS-\{0, 7\}$ and $c=s_2 s_4 s_5 s_6$. Then $w^{S_0-\{7\}}_0w^{S_0}_0 c$ is conjugate to a Coxeter element of $W_0$.
\

%Now we give a proof to Lemma 2.2
%\begin{proof}
%If $\t\in\Om'$ and $\t\neq1$, then $\t$ corresponds uniquely to a minuscule coweight. We have minuscule coweight(s) only in those types discussed in $\S 3.2$. For each case, we just drop one $\t$-orbit $S'$ on $\tS$. Then we take $J=\tS\backslash S'$, and $c_J=t^{\o^\vee_i} w_0^{I-\{i\}}w_0^Ic$ or $c_J=t^{\o^\vee_i} w_0^{I-\{i\}}w_0^I\d(c)\d$ where $c$ is taken in $\S 3.2$ for each case.
%\end{proof}

\section{Proof of the main theorem}

\subsection{} We keep the notation in section 1. For any $w, w' \in \tW'$ and $i \in \tS$, we write $w \xrightarrow{s_i}w'$ if $w'=s_i w s_i$ and $\ell(w') \le \ell(w)$. We write $w \to w'$ if there is a sequence of $w=w_0, w_1, \cdots, w_n=w'$ of elements in $\tW'$ such that for any $k\in\{0, 1, \cdots, n-1\}$, $w_k \xrightarrow{s_i} w_{k+1}$ for some $i \in \tS$. We write $w \approx w'$ if $w \to w'$ and $w' \to w$.

The following result is proved in \cite{HN}.

\begin{thm}\label{HN}
Let $\co$ be a $W_a$-conjugacy class of $\tW'$ with finite Coxeter part and $\co_{\min}$ be the set of minimal length elements in $\co$. Then for any $w \in \co$ and  $w' \in \co_{\min}$, $w \to w'$.
\end{thm}

\subsection{} Let $\co$ be a $W_a$-conjugacy class of $\tW'$ with finite Coxeter part and let $\t\in\Om'$ with $\co\subset W_a\t$. In section 2, we have proved that there exists a maximal proper $\t$-stable subset $J$ of $\tS$ and a Coxeter element $c_J$ of $W_J \rtimes \<\t\>$ such that $c_J \in \co$.

Let $w$ be a minimal length element in $\co$. By Theorem \ref{HN}, $c_J \to w$. Since $c_J$ is a Coxeter element of $W_J \rtimes \<\t\>$, $w$ is also a Coxeter element of $W_J \t \subset W_J \rtimes \<\t\>$ and $c_J \approx w$.

Since $J$ is a proper subset of $\tS$, $W_J \rtimes \<\t\>$ is a finite group. Hence any two Coxeter element of $W_J \t$ are conjugated by an element of $W_J$. Thus all the Coxeter elements of $W_J \t$ are contained in $\co$.

Therefore $\co_{\min}$ is the set of Coxeter elements in $W_J \t \subset W_J \rtimes \<\t\>$.

Moreover, $J=\supp(w)$ for any $w \in \co_{\min}$. This proves the uniqueness of $J$.

%From section 2 and for any $w\in\co$, there exists a maximal $\t$-stable subset $J\subset\tS$ and a Coxeter element $w'\in W_J\rtimes\langle\t\rangle$, such that $w\to w'$. Moreover, if $w'$ and $w''$ are Coxeter elements of $W_J\rtimes\langle\t\rangle$, then $w'\approx w''$. Then we can deduce that $\co_{min}=$\{Coxeter elements of $W_J\rtimes\langle\t\rangle$\}.
%\subsection{}If $J'$ is another maximal $\t$-stable proper subset of $\tS$ satisfies Theorem 1.1, then \{Coxeter elements of $W_{J'}\rtimes\langle\t\rangle$\}=$\co_{min}$=\{Coxeter elements of $W_J\rtimes\langle\t\rangle$\}. The identities together with $J$ and $J'$ are maximal $\t$-stable proper subset of $\tS$ deduce that $J'=J$.

\bibliographystyle{amsalpha}

\end{document}